\documentclass[dvipsnames,a4paper,twoside]{article}
\usepackage{amsthm,amsmath,amssymb,xcolor,soul,xspace}
\usepackage[bookmarks=true,bookmarksopen=true]{hyperref}

\usepackage[ruled,algo2e]{algorithm2e}

\theoremstyle{plain}
\newtheorem{theorem}{Theorem}[section]
\newtheorem{lemma}[theorem]{Lemma}
\newtheorem{corollary}[theorem]{Corollary}

\theoremstyle{definition}

\theoremstyle{remark}
\newtheorem{remark}{Remark}

\newcommand{\pd}{\hat p}
\newcommand{\qd}{\hat q}

\newcommand{\dx}{\,\mathrm{d}x}
\newcommand{\dt}{\,\mathrm{d}t}
\newcommand{\supp}{\operatorname{supp}}

\newcommand{\Pb}{\mbox{\rm (P)}\xspace}

\newcommand{\uad}{U_{\rm ad}}

\title{New second order sufficient optimality conditions for state constrained parabolic control problems\thanks{The first two authors were partially supported by MCIN/ AEI/10.13039/501100011033 under research project PID2020-114837GB-I00.}}

\author{Eduardo Casas\thanks{Departamento de Matem\'{a}tica Aplicada y Ciencias de la Computaci\'{o}n, E.T.S.I. Industriales y de Telecomunicaci\'on, Universidad de Cantabria, 39005 Santander, Spain
(\texttt{eduardo.casas@unican.es})},
\and Mariano Mateos\thanks{Departamento de Matem\'{a}ticas, Campus de Gij\'on, Universidad de Oviedo, 33203, Gij\'on, Spain(\texttt{mmateos@uniovi.es})},
\and Arnd R\"osch\thanks{Fakult\"at f\"ur Mathematik, Universt\"at Duisburg-Essen, D-45127 Essen, Germany(\texttt{arnd.roesch@uni-due.de})}
}

\date{}
\ifpdf
\hypersetup{
  pdftitle={New second order sufficient optimality conditions for state constrained parabolic control problems},
  pdfauthor={E.~Casas, M.~Mateos, and A.~R\"osch}
}
\fi

\begin{document}

\maketitle

\begin{abstract}
We study a control problem governed by a semilinear parabolic equation with pointwise control and state constraints imposed at every point of the space-time cylinder. We obtain second order sufficient optimality conditions for local optimality in the sense of $L^2(Q)$. Our results are valid for spatial domains of dimension less than or equal to three.
\end{abstract}

\begin{quote}
\textbf{Keywords:}
optimal control,  semilinear partial differential equations, optimality conditions, state constraints, Borel measures
\end{quote}

\begin{quote}
\textbf{AMS Subject classification: }
35K58, 
35R06, 
49K20  
\end{quote}

\section{Introduction}
\label{S1}
Let us consider a domain $\Omega\subset\mathbb R^n$, $n\leq 3$, with a Lipschitz boundary $\Gamma$. Given a finite horizon $T>0$ we denote $Q=\Omega\times(0,T)$ and $\Sigma = \Gamma\times(0,T)$. In this paper, we investigate  second order sufficient optimality conditions for the control problem
\[
\Pb\qquad \min_{u \in \uad}  J(u):=\int_Q L(x,t,y_u(x,t)) \dx\dt+ \frac{\nu}{2}\int_Q u(x,t)^2\dx\dt,
\]
where $L:Q\times \mathbb{R}\to\mathbb{R}$ is a given function, $\nu > 0$, and
\[
\uad =\{u \in L^\infty(Q): \alpha \le u(x,t) \le \beta\ \text{ and }y_u(x,t)\leq\gamma\text{ for a.a. } (x,t) \in Q\}
\]
with
$\gamma >0$ and $-\infty<\alpha<\beta<+\infty$ if $n=2$ or $n=3$, or  $-\infty\leq\alpha<\beta\leq+\infty$ if $n=1$.

Above $y_u$ denotes the state associated to the control $u$ related by the following semilinear parabolic state equation
\begin{equation}\label{E01}
\left\{\begin{array}{rcll}
\displaystyle\frac{\partial y_u}{\partial t} + A y_u + f(x,t,y_u) &=& u&\mbox{ in }Q,\\
 y_u& =& 0 &\mbox{ on }\Sigma,\\
  y_u(0)& = &y_0&\mbox{ in }\Omega.
  \end{array}
  \right.
\end{equation}
Assumptions on the data $A$, $f$, $y_0$, and $y_d$ are specified in Section \ref{S2}.

The goal of this paper is to obtain some second order sufficient optimality conditions for problem (P).
Although there is an extensive literature about sufficient optimality conditions for pointwise state constrained control problems governed by ordinary differential equations --see e.g. 
\cite{Bonnans-Hermant2009c,Bonnans-Hermant2009a,Bonnans-Hermant2009b,Malanowski2007,Malanowski2008,Malanowski-Maurer-Pickenhain2004,PickenhainTammer1991}
just to cite a few significant references--,
there are few papers  treating this topic for problems governed by partial differential equations.

In the elliptic case, as far as we know, the first work was published in 2000 by Casas, Tr\"oltzsch and Unger \cite{Casas-Troltz-Unger2000}. In this paper, optimality is obtained in the $L^\infty(\Omega)$ sense under a condition of positive-definiteness of the second derivative of the Lagrangian on an appropriately extended cone.
It is well known that the use of extended cones in infinite-dimensional optimization problems is necessary in many circumstances; see the works by  Maurer and Zowe \cite{Maurer-Zowe79},  Malanowski \cite{Malanowski1993} or Dunn \cite{Dunn98}. The extension used in \cite{Casas-Troltz-Unger2000} is based on the one introduced in \cite{Maurer-Zowe79}, combining this approach with a detailed splitting technique. In 2008, Casas, De los Reyes and Tr\"oltzsch \cite{Casas-delosReyes-Troltz-2008-SIOPT} obtained local optimality in the $L^2(\Omega)$ sense. They used a non-extended cone that seems to be optimal. While in the \cite{Casas-Troltz-Unger2000,Casas-delosReyes-Troltz-2008-SIOPT} the authors assumed $\nu>0$, in 2014 Casas and Tr\"oltzsch \cite{Casas-Troltzsch2014} derived sufficient optimality conditions including the case $\nu = 0$ and an objective functional promoting sparsity of the optimal controls. To do this they considered an extension of the cone given in \cite{Casas-delosReyes-Troltz-2008-SIOPT} which was different from the one in \cite{Casas-Troltz-Unger2000}. 

Let us comment now on the existing results for the parabolic case. In 2000, Raymond and Tr\"oltzsch \cite{Raymond-Troltz2000} obtained  a result for $n=1$ using an extended cone in the spirit of \cite{Maurer-Zowe79}. In \cite{Casas-delosReyes-Troltz-2008-SIOPT} optimality conditions were obtained with a non-extended cone, also for $n=1$. In both papers, local optimality was only obtained in the $L^\infty(Q)$ sense and $\nu$ strictly positive was assumed.
In the work at hand, we prove that the second order condition given in \cite{Casas-delosReyes-Troltz-2008-SIOPT} is sufficient for local optimality in the $L^2(Q)$ sense.
The methods used for $n=1$ cannot be applied for dimension $n>1$ because the control-to-state mapping is not continuous from $L^2(Q)$ into $C(\bar Q)$ if $n>1$. This difficulty was overcome by Krumbiegel and Rehberg in 2013  \cite{Krumbiegel-Rehberg2013} for $\nu>0$ and $n=2$ or $3$ by considering an extended cone depending on two parameters $\beta>0$ and $\tau>0$. They proved local optimality in the $L^\infty(Q)$ sense under their second order condition. To this end, a splitting technique similar to the one employed in \cite{Casas-Troltz-Unger2000} is used.
Using a different cone, we give a sufficient second order condition leading to local optimality in the $L^2(Q)$ sense for $n=2$ or $3$ and $\nu>0$; see Theorem \ref{T4.1}. Our extended cone $C^\tau$ is a translation to the parabolic case of the one introduced in \cite{Casas-Troltzsch2014}. It depends on a parameter $\tau>0$ that can be arbitrarily small and it converges in a decreasing way to the non-extended cone $C$ as $\tau$ tends to zero.

The local optimality in $L^2(Q)$ is the main novelty of our work. $L^2(Q)$ local optimality is much stronger than $L^\infty(Q)$ local optimality and it is more useful to study the stability of the optimal control with respect to small perturbations of the data of control problem, to obtain error estimates for the numerical discretization of the problem and to analyze the convergence of the optimization algorithms to solve the problem; see \cite{Casas-Troltzsch2015}. Due to the presence of control constraints, $L^2(Q)$-local optimality implies the $L^p(Q)$ local optimality for every $p\in[1,+\infty]$; see Section \ref{S3}.

A key point in the proof of the second order conditions in \cite{Casas-Troltzsch2014} for the case $\nu = 0$ was the demonstration of the boundedness of the adjoint state despite the fact that it depends on a Borel measure. The reader is also referred to \cite{CMV2014} for the proof of the boundedness of the adjoint state and its use to get error estimates in the numerical approximation of the control problem. We have not been able to prove the corresponding boundedness for the adjoint state in our control problem \Pb. Thus, the derivation of second order conditions for the case $\nu = 0$ under the presence of pointwise state constraints remains open in the parabolic case; see remarks \ref{ZZR2} and \ref{R3}.

We want also to emphasize that our problem suffers the two-norm discrepancy for $n>1$. The reader is referred to the papers by Ioffe \cite{Ioffe1979}, Maurer \cite{Maurer1981}, or  Malanowski \cite{Malanowski1993} for the treatment of this topic in optimization problems or control problems governed by ordinary differential equations. As far as we know, the first reference dealing with this issue in a problem governed by partial differential equations is the 1993 paper by Goldberg and Tr\"oltzsch \cite{Goldberg-Troltz1993}. The reader is referred to the 2012 paper by Casas and Tr\"oltzsch \cite{Casas-Troltzsch2012} for a systematic approach and further references.

The plan of the paper is as follows. In Section \ref{S2} we state precisely the assumptions, recall the regularity results that we are going to use, and obtain the differentiability properties of the Lagrangian. First order optimality conditions are formulated in Section \ref{S3}. The main result of the paper consists in a second order sufficient condition for optimality in the sense of $L^2(Q)$. This is proved in Section \ref{S4}. Finally, in Section \ref{S5}, we comment on the case of bilateral constraints.

%

\section{Assumptions and study of the equations and the functional}\label{S2}

We make the following assumptions on the data of the problem.

(A1) \label{A1} $A$ denotes the elliptic operator
\[
Ay =-\sum_{i,j=1}^n \partial_{x_j}(a_{i,j}(x)\partial_{x_i} y) + \sum_{j = 1}^n b_j(x,t)\partial_{x_j} y,
\]
where $b_j\in L^\infty(Q)$, $a_{i,j}\in L^\infty(\Omega)$, and the uniform ellipticity condition
\begin{equation}\label{E02}
\exists \lambda_A>0 : \lambda_A|\xi|^2\leq \sum_{i,j=1}^n a_{i,j}(x)\xi_i\xi_j\ \mbox{ for all }\xi\in\mathbb R^n\mbox{ and a.a. }x\in \Omega
\end{equation}
holds.

(A2) \label{A2} We assume that $L:Q\times\mathbb R\to\mathbb R$ is a Carath\'eodory function of class $C^2$ with respect to the last variable satisfying the following properties:
\begin{equation}
\left\{\begin{array}{l}
L(\cdot,\cdot,0)\in L^{1}(Q),\vspace{2mm}\\
\displaystyle\forall M>0\ \exists \psi_{L,M}\in L^1(Q) : \left|\frac{\partial L}{\partial y}(x,t,y)\right|\leq \psi_{L,M}(x,t)\ \forall |y|\leq M,\\
\displaystyle\forall M>0\ \exists \Psi_{L,M}\in L^{5/4}(Q) : \left|\frac{\partial^2 L}{\partial y^2}(x,t,y)\right|\leq \Psi_{L,M}(x,t)\ \forall |y|\leq M,\\
\begin{array}{l}\forall \epsilon>0 \text{ and } \forall M>0\ \exists \delta>0 \text{ such that}\vspace{2mm}\\
\displaystyle\left|\frac{\partial^2 L}{\partial y^2}(x,t,y_1)-\frac{\partial^2 L}{\partial y^2}(x,t,y_2)\right|<\epsilon\ \ \forall|y_1|,|y_2|\leq M \text{ with } \ |y_1-y_2|<\delta,\end{array}
\end{array}\right.
\label{E03}
\end{equation}
for almost all $(x,t) \in Q$.

(A3) \label{A3} We assume that $f:Q\times\mathbb R\to\mathbb R$ is a Carath\'eodory function of class $C^2$ with respect to the last variable satisfying the following properties:
\begin{equation}
\left\{\begin{array}{l}
\displaystyle\exists  C_f \in\mathbb R  : \frac{\partial f}{\partial y}(x,t,y)\geq C_f\ \forall y\in\mathbb R,\vspace{2mm}\\
f(\cdot,\cdot,0)\in L^{\pd}(0,T;L^{\qd}(\Omega)) \ \text{ for some } \pd,\qd\geq 2 \text{ with } \frac{1}{\pd} + \frac{n}{2\qd}<1,\vspace{2mm}\\
\displaystyle\forall M>0\ \exists C_{f,M}>0 : \left|\frac{\partial^j f}{\partial y^j}(x,t,y)\right|\leq C_{f,M}\ \forall |y|\leq M\mbox{ and } j=1,2,\vspace{2mm}\\
\begin{array}{l}\forall \epsilon>0 \text{ and } \forall M>0\ \exists \delta>0 \text{ such that}\vspace{2mm}\\
\displaystyle\left|\frac{\partial^2 f}{\partial y^2}(x,t,y_1)-\frac{\partial^2 f}{\partial y^2}(x,t,y_2)\right|<\epsilon\ \ \forall|y_1|,|y_2|\leq M \text{ with } \ |y_1-y_2|<\delta,\end{array}
\end{array}\right.
\label{E04}
\end{equation}
for almost all $(x,t) \in Q$.
Examples of functions $f$ satisfying the above assumptions are the polynomials of odd degree with positive leading coefficients or the exponential function $f(x,t,y)= g(x,t)\mathrm{exp}(y)$ with $g\in L^\infty(Q)$, $g(x,t)\geq 0$ for almost all $(x,t) \in Q$.

(A4)\label{A4} For the initial datum we assume
$y_0 \in C_0(\Omega)$,  where $C_0(\Omega)$ denotes the space of continuous function in $\bar\Omega$ vanishing on $\Gamma$. We also assume that  $y_0(x)< \gamma$ for all $x\in\Omega$.

(A5) \label{A5} In the rest of the paper $p$ is a fixed number such that
$p > 1+ \frac{n}{2}$ if $n = 2$ or 3, and $p \ge 2$ if $n = 1$.

Next, we collect several results concerning the regularity properties of the solutions of the equations involved in this problem.

\begin{theorem}\label{T2.1}
For every $u\in L^{p}(Q)$ there exists a unique solution $y_u$ of \eqref{E01}, belonging to the space $L^2(0,T;H^1_0(\Omega))\cap C(\bar Q)$. Moreover, there exist a monotone non-decreasing function $\eta:[0,\infty)\to [0,+\infty)$ and a positive constant $K$ such that
\begin{align*}
 &\|y_u\|_{C(\bar Q)}\leq  \eta(\|u\|_{L^{p}(Q)}+\|f(\cdot,\cdot, 0)\|_{L^{\pd}(0,T;L^{\qd}(\Omega))} + \|y_0\|_{C(\bar\Omega)} ),\\
&\|y_u\|_{L^2(0,T;H^1_0(\Omega))}\leq  K(\|u\|_{L^{2}(Q)}+\|f(\cdot,\cdot, 0)\|_{L^{2}(0,T;L^{2}(\Omega))} + \|y_0\|_{L^2(\Omega)} ).
\end{align*}
Moreover, for all $R>0$ there exists $C_{p,R}>0$ such that $\|y_u-y_{\bar u}\|_{C(\bar Q)}\leq  C_{p,R}\|u-\bar u\|_{L^{p}(Q)}$ for all $u,\bar u\in B_R(0)\subset L^p(Q)$.
Finally, if $u_k\rightharpoonup u$ weakly in  $L^{p}(Q)$, then the strong convergence \[
\|y_{u_k}-y_u\|_{C(\bar Q)}+\|y_{u_k}-y_u\|_{L^2(0,T;H^1_0(\Omega))} \to 0\]
holds.
\end{theorem}
\begin{proof}
For the proof of this theorem the reader is referred to \cite[Theorem 2.1]{Casas-Kunisch2022} and \cite[Theorem 2.1]{CM-SIOPT-2020}. In these papers, the proofs are based on the $L^\infty$ and H\"older estimates for the solutions of parabolic equations; see \cite[Chapter 3]{LSU} and \cite{DER2017}.
\end{proof}

From Theorem \ref{T2.1} we infer that the mapping $G:L^p(Q)\to L^2(0,T;H_0^1(\Omega)) \cap C(\bar Q)$ given by $G(u) = y_u$ is well defined. Let us study its differentiability properties. A proof of the next result can be found in
\cite[Theorem 5.1]{Casas-Troltzsch2012}.

\begin{theorem}\label{T2.2}
  The mapping $G:L^p(Q)\to L^2(0,T;H_0^1(\Omega)) \cap C(\bar Q)$ is of class $C^2$. For $u,v\in L^p(Q)$, the derivative $G'(u)v = z_{u,v}$ is the solution of  the equation
\begin{equation}\label{E05}
\left\{\begin{array}{rcll}
\displaystyle\frac{\partial z}{\partial t} + A z+ \frac{\partial f}{\partial y} (x,t,y_u) z &=& v&\mbox{ in }Q,\\
 z& =& 0 &\mbox{ on }\Sigma,\\
  z(0)& = &0&\mbox{ in }\Omega.
  \end{array}
  \right.
\end{equation}
For $u,v_1,v_2\in L^p(Q)$, the second derivative $G''(u)(v_1,v_2) = w_{u,v_1,v_2}$ is the solution of the equation
\begin{equation}\label{E06}
\left\{\begin{array}{rcll}
\displaystyle\frac{\partial w}{\partial t} + A w+ \frac{\partial f}{\partial y} (x,t,y_u) w &=& -\displaystyle\frac{\partial^2 f}{\partial y^2}(x,t,y_u)z_{u,v_1} z_{u,v_2}  &\mbox{ in }Q,\\
 w& =& 0 &\mbox{ on }\Sigma,\\
  w(0)& = &0&\mbox{ in }\Omega.
  \end{array}
  \right.
\end{equation}
\end{theorem}

The following estimates of $z_{u,v}$ will be useful along the paper.
\begin{lemma}\label{L2.3}For all $u\in L^p(Q)$ there exist constants $C_2$ and $C_p$ such that the following inequalities hold:
\[\|z_{u,v}\|_{L^{10}(Q)}\leq C_2\|v\|_{L^2(Q)}\ \forall v\in L^2(Q),
\mbox{ and }\|z_{u,v}\|_{L^{\infty}(Q)}\leq C_p\|v\|_{L^p(Q)}\ \forall v\in L^p(Q).\]
\end{lemma}
These inequalities follow from Theorem 2.3 and Lemma 2.3, respectively, of \cite{Casas-Wachsmuth2022}.
The independence of $u$ follows from Theorem 2.1 and Assumption (A3).

Let $a_0$ be a function in $L^\infty (Q)$. For $(\mu_Q,\mu_\Omega) \in \mathcal{M}(Q) \times \mathcal{M}(\Omega)$, where $\mathcal{M}(Q)$ and $\mathcal{M}(\Omega)$ denote the spaces of real and regular Borel measures in $Q$ and $\Omega$ respectively, we consider the equation
\begin{equation}\label{E07}
\left\{\begin{array}{rcll}
-\displaystyle\frac{\partial \varphi}{\partial t} + A^* \varphi + a_0(x,t)\varphi &=& \mu_Q&\mbox{ in }Q,\\
 \varphi& =& 0 &\mbox{ on }\Sigma,\\
  \varphi(T)& = &\mu_\Omega&\mbox{ in }\Omega.
  \end{array}
  \right.
 \end{equation}

Following \cite[Definition 2.1]{Casas-Kunisch2016}, we say that $\varphi\in L^1(Q)$ is a solution of \eqref{E07} if it satisfies
\begin{equation}\label{E08}
  \int_Q \varphi \left(\frac{\partial y}{\partial t}+Ay +a_0 y\right)\dx\dt = \int_Q y\mathrm{d}\mu_Q + \int_Q y(T)\mathrm{d}\mu_\Omega\ \forall y\in Y,
\end{equation}
where
\[
Y = \left\{y\in L^2(0,T;H^1_0(\Omega)):\ \frac{\partial y}{\partial t}+Ay +a_0 y\in L^\infty(Q),\ y(0)=0\right\}.
\]
Let us notice that $Y \subset C(\bar Q)$, hence the right hand side of \eqref{E08} is well defined.
The following result is proved in \cite[Theorem 2.2]{Casas-Kunisch2016}.

\begin{lemma}\label{L2.4}
 For every $(\mu_Q,\mu_\Omega) \in \mathcal{M}(Q) \times \mathcal{M}(\Omega)$, there exists a unique $\varphi$ solution of \eqref{E07}
that belongs to $L^r(0,T;W^{1,s}(\Omega))$ for all $r,s\in [1,2)$ with $2/r+n/s>n+1$. Moreover, for a constant $C_{r,s}>0$, that may depend on $\|a_0\|_{L^\infty(Q)}$ but is independent of $(\mu_Q,\mu_\Omega)$ the following estimate holds.
\[\|\varphi \|_{L^r(0,T;W^{1,s}(\Omega))}\leq C_{r,s} (\|\mu_{Q}\|_{\mathcal{M}(Q)}+\|\mu_\Omega\|_{\mathcal{M}(\Omega)}).\]
\end{lemma}

\begin{lemma}\label{L2.5}
Under the same assumptions as in Lemma \ref{L2.4}, we also have that
$\varphi\in L^q(Q)$ for all $q < (n+2)/n$
and there exists a constant $K_q>0$  such that
 \[\|\varphi\|_{L^q(Q)}\leq K_q (\|\mu_Q\|_{\mathcal{M}(Q)}+\|\mu_\Omega\|_{\mathcal{M}(\Omega)}).\]
 \end{lemma}

\begin{proof}
Consider $f\in C^\infty_0(Q)$, the space of infinitely differentiable functions with compact support in $Q$, and let $y\in L^2(0,T;H^1_0(\Omega))\cap C([0,T];L^2(\Omega))$ be the unique solution of
\[\left\{\begin{array}{rcll}
\displaystyle\frac{\partial y}{\partial t} + A y + a_0(x,t) y &=& f&\mbox{ in }Q,\\
 y& =& 0 &\mbox{ on }\Sigma,\\
  y(0)& = & 0 &\mbox{ in }\Omega.
  \end{array}
  \right.
  \]
  From the classical regularity results in \cite[Chapter 3, \S7]{LSU} and \cite{DER2017} and the condition $q<(n+2)/n$, or equivalently $q'>1+n/2$, we deduce the existence of a constant $K_q$, that may depend on $\|a_0\|_{L^\infty(Q)}$, such that $\|y\|_{C(\bar Q)}\leq K_q \|f\|_{L^{q'}(Q)}$.
 It is clear that $y\in Y$, and from \eqref{E08} we obtain
\begin{align*}
 \int_Q \varphi &  f\dx\dt =   \int_Q \varphi\left(\displaystyle\frac{\partial y}{\partial t} + A y + a_0(x,t) y\right)\dx\dt   =   \int_Q y \mathrm{d}\mu_Q +\int_\Omega y(T) \mathrm{d}\mu_\Omega\\
  \leq  & \|y\|_{C(\bar Q)} (\|\mu_Q\|_{\mathcal{M}(Q)} + \|\mu_\Omega\|_{\mathcal{M}(\Omega)})
  \leq  K_q \|f\|_{L^{q'}(Q)} (\|\mu_Q\|_{\mathcal{M}(Q)} + \|\mu_\Omega\|_{\mathcal{M}(\Omega)}).
\end{align*}
The result follows from the density of $C^\infty_0(Q)$ in $L^{q'}(Q)$.
\end{proof}
\begin{remark}
If $n=2$ or $n=3$, Lemma \ref{L2.5} is a consequence of Lemma \ref{L2.4} and the embedding $L^{q}(0,T;W^{\frac{qn}{n+q}}(\Omega))\hookrightarrow L^q(\Omega)$. For $n=1$, nevertheless, this embedding would only lead to $\varphi\in L^q(Q)$ for all $q<2$; see \cite[Remark 2.4]{Casas-Kunisch2016}.
\end{remark}

 As a consequence of Lemma \ref{L2.3} and the chain rule, we obtain the differentiability properties of $J:L^p(Q) \longrightarrow \mathbb{R}$.
\begin{corollary}\label{C2.6}The functional $J:L^p(Q)\to \mathbb{R}$ is of class $C^2$. For every  $u, v\in L^p(Q)$, its derivatives are given by
\[J'(u)v = \int_Q(\varphi+\nu u)v\dx\dt\]
and
\[J''(u)v^2 = \int_Q\left[\left(\frac{\partial^2 L}{\partial y^2}(x,t,y_u)-\frac{\partial^2 f}{\partial y^2}(x,t,y_u)\varphi\right)z_{u,v}^2+\nu v^2\right]\dx\dt,\]
where $ \varphi\in L^r(0,T;W^{1,s}(\Omega))\cap L^q(Q)$ for all $r,s\in [1,2)$ with $2/r+n/s>n+1$ and all $q<(n+2)/n$ is the solution  of
\[\left\{\begin{array}{rcll}
-\displaystyle\frac{\partial\varphi}{\partial t} + A^* \varphi + \displaystyle\frac{\partial f}{\partial y}(x,t, y_u)\varphi &=&  \displaystyle\frac{\partial L}{\partial y}(x,t, y_u) & \mbox{ in }Q,\\
 \varphi & =& 0&\mbox{ on }\Sigma,\\
 \varphi(\cdot,T) & = & 0&\mbox{ in }\Omega.
  \end{array}
  \right.
\]
\end{corollary}

Now, we  introduce the Lagrangian function associated with problem \Pb:
\begin{align*}
&\mathcal{L}:L^p(Q) \times \mathcal{M}(Q) \times \mathcal{M}(\Omega) \longrightarrow \mathbb{R}\\
&\mathcal{L}(u,\mu_Q,\mu_\Omega) = J(u) + \int_Q (y_u-\gamma) \mathrm{d}\mu_Q + \int_\Omega (y_u(\cdot,T)-\gamma) \mathrm{d}\mu_\Omega,
\end{align*}

\begin{corollary}\label{C2.7}The Lagrangian function $\mathcal{L}$ is of class $C^2$. For every $u, v \in L^p(Q)$ and $\mu =(\mu_Q,\mu_\Omega) \in \mathcal{M}(Q) \times \mathcal{M}(\Omega)$, its derivatives are given by
\[\frac{\partial \mathcal{L}}{\partial u}(u,\mu_Q,\mu_\Omega)v = \int_Q(\varphi_{u,\mu}+\nu u)v\dx\dt\]
and
\[\frac{\partial^2 \mathcal{L}}{\partial u^2}(u,\mu_Q,\mu_\Omega)v^2 = \int_Q\left[\left(\frac{\partial^2L}{\partial y^2}(x,t, y_u) -\frac{\partial^2 f}{\partial y^2}(x,t,y_u)\varphi_{u,\mu}\right)z_{u,v}^2+\nu v^2\right]\dx\dt,\]
where $\varphi_{u,\mu}\in L^r(0,T;W^{1,s}(\Omega))\cap L^q(Q)$ for all $r,s\in [1,2)$ with $2/r+n/s>n+1$ and all $q<(n+2)/n$ is the solution of
\begin{equation}\label{E09}
\left\{\begin{array}{rcll}
-\displaystyle\frac{\partial\varphi}{\partial t} + A^* \varphi + \displaystyle\frac{\partial f}{\partial y}(x,t, y_u)\varphi &=&  \displaystyle\frac{\partial L}{\partial y}(x,t, y_u) + \mu_Q& \mbox{ in }Q,\\
 \varphi& =& 0&\mbox{ on }\Sigma,\\
 \varphi(\cdot,T) & = & \mu_\Omega&\mbox{ in }\Omega.
  \end{array}
  \right.
\end{equation}
\end{corollary}

The following result is an immediate consequence of Lemma \ref{L2.4}, Lemma \ref{L2.5} and Theorem \ref{T2.1}.

\begin{lemma}\label{L2.8}
For any $u\in L^p(Q)$ and $(\mu_Q,\mu_\Omega) \in \mathcal{M}(Q) \times \mathcal{M}(\Omega)$  we have that $\varphi_{u,\mu}\in L^r(0,T;W^{1,s}(\Omega)) \cap L^q(Q)$ for all $r,s\in [1,2)$ with $2/r+n/s>n+1$ and all $q<(n+2)/n$. Moreover, there exists a constant $M_q$ depending continuously on $\|y_u\|_{L^\infty(Q)}$ such that the following estimate holds:
\[
\|\varphi_{u,\mu}\|_{L^q(\Omega)}\leq M_q(1 + \|\mu_Q\|_{\mathcal{M}(Q)} + \|\mu_\Omega\|_{\mathcal{M}(\Omega)}).
\]
\end{lemma}

One of the keys in the proof of second order sufficient optimality conditions is the continuity of the second derivative of the Lagrangian in an appropriate sense. In the next theorem, we establish the result needed in our case.

\begin{theorem}\label{T2.9}
Consider $\bar u\in L^p(Q)$, $\bar y $ its associated state, and measures $(\bar\mu_Q,\bar\mu_\Omega) \in \mathcal{M}(Q) \times \mathcal{M}(\Omega)$. Then, for every $\varepsilon >0$ there exists $\rho>0 $ such that for all $u\in L^p(Q)$ with $\|y_u-\bar y\|_{C(\bar Q)} <\rho$ the following inequality holds:
\[\left|\frac{\partial^2\mathcal{L}}{\partial u^2}(\bar u, \bar\mu_Q,\bar\mu_\Omega)v^2-\frac{\partial^2\mathcal{L}}{\partial u^2}( u, \bar\mu_Q,\bar\mu_\Omega)v^2\right|\leq \varepsilon \|v\|_{L^2(Q)}^2\quad \forall\ v\in L^2(Q).\]
\end{theorem}

Along the proof, and in the rest of the paper, we will denote, $\bar\varphi = \varphi_{\bar u,\bar\mu}$, $\varphi_u = \varphi_{u,\bar\mu}$ and $z_v = z_{\bar u,v}$.

\begin{proof}[Proof of Theorem \ref{T2.9}]
Consider a fixed $\varepsilon>0$ and let $\rho>0$  be as defined in \eqref{E12} below. Set $M_\infty = 1+ \|\bar y\|_{C(\bar Q)}$. Take $u\in L^p(Q)$ such that $\|y_u-\bar y\|_{C(\bar Q)} <\rho$.
We split the difference to be estimated into five parts:

\begin{align*}
  \bigg|\frac{\partial^2\mathcal{L}}{\partial u^2}&(\bar u, \bar\mu_Q,\bar\mu_\Omega)v^2-\frac{\partial^2\mathcal{L}}{\partial u^2}( u, \bar\mu_Q,\bar\mu_\Omega)v^2\bigg|  \\
  \leq  & \int_Q \left| \left(\frac{\partial^2 L}{\partial y^2}(x,t,\bar y) -\frac{\partial^2 f}{\partial y^2}(x,t,\bar y) \bar\varphi \right)z_{v}^2
  \right. \\
   & \qquad \left.
   -
  \left(\frac{\partial^2 L}{\partial y^2}(x,t,y_u)-\frac{\partial^2 f}{\partial y^2}(x,t,y_u)\varphi_u\right)z_{u,v}^2\right| \dx\dt\\
  \leq &
  \int_Q\left|\frac{\partial^2 L}{\partial y^2}(x,t,\bar y)(z_{u,v}^2-z_v^2)\right|\dx\dt \\
  & +  \int_Q\left| \left(\frac{\partial^2 L}{\partial y^2}(x,t,\bar y)-\frac{\partial^2 L}{\partial y^2}(x,t,y_u)\right)z_{u,v}^2\right|\dx\dt \\
  & + \int_Q \left| \frac{\partial^2 f}{\partial y^2}(x,t,\bar y)\bar\varphi \left(z_{u,v}^2-z_v^2\right)\right|\dx\dt \\
  & + \int_Q\left|\left(\frac{\partial^2 f}{\partial y^2}(x,t,y_u)-\frac{\partial^2 f}{\partial y^2}(x,t,\bar y)\right)\bar\varphi z_{u,v}^2\right|\dx\dt\\
  & + \int_Q\left| \frac{\partial^2 f}{\partial y^2}(x,t, y_u) (\varphi_u-\bar\varphi) z_{u,v}^2\right|\dx\dt 
  =
  I_1+I_2+I_3+I_4+I_5.
\end{align*}
Let us estimate each of the summands  $I_i$.
For the first one, applying Theorem \ref{T2.1}, Assumption(A2) and H\"older's inequality we obtain
\[I_1 \leq \|\Psi_{L,M_\infty}\|_{L^{5/4}(Q)}\|z_{u,v}-z_v\|_{L^{10}(Q)}\|z_{u,v}+z_v\|_{L^{10}(Q)}. \]
Using Lemma \ref{L2.3}, we obtain
\begin{equation}\label{E10}\|z_{u,v}+z_v\|_{L^{10}(Q)}\leq 2C_2\|v\|_{L^2(Q)}.\end{equation}

Denote $\eta = z_{u,v}-z_v$. This function satisfies the linear equation
\[
\left\{\begin{array}{rcll}
\displaystyle\frac{\partial \eta}{\partial t} + A \eta+ \frac{\partial f}{\partial y} (x,t,\bar y) \eta &=& \displaystyle\left(\frac{\partial f}{\partial y} (x,t,\bar y)-\frac{\partial f}{\partial y} (x,t,y_u)\right)z_{u,v}&\mbox{ in }Q,\\
 \eta& =& 0 &\mbox{ on }\Sigma,\\
  \eta(0)& = &0&\mbox{ in }\Omega.
  \end{array}
  \right.
\]
By the Mean Value Theorem, there exists a measurable function $\theta:Q\to [0,1]$ such that
\[
\left\{\begin{array}{rcll}
\displaystyle\frac{\partial \eta}{\partial t} + A \eta+ \frac{\partial f}{\partial y} (x,t,\bar y) \eta &=& \displaystyle\frac{\partial^2 f}{\partial y^2} (x,t, y_\theta)(\bar y -y_u)z_{u,v}&\mbox{ in }Q,\\
 \eta& =& 0 &\mbox{ on }\Sigma,\\
  \eta(0)& = &0&\mbox{ in }\Omega,
  \end{array}
  \right.
\]
where $y_\theta = \bar y + \theta(y_u-\bar y)$.
Using Lemma \ref{L2.3}, Assumption (A3), Theorem \ref{T2.1}, and H\"older's inequality we have that
\begin{align}
\|z_{u,v}-z_v\|_{L^{10}(Q)}
\leq &
C_2 C_{f,M_\infty} \|y_u-\bar y\|_{L^\infty(Q)} \|z_{u,v}\|_{L^2(Q)} \nonumber \\
\leq &
|Q|^{2/5}C_2 C_{f,M_\infty}\rho \|z_{u,v}\|_{L^{10}(Q)} \label{E11}
\leq
 |Q|^{2/5}C_2^2C_{f,M_\infty}\rho\|v\|_{L^2(Q)}.
\end{align}
Gathering these estimates, we get
\[I_1\leq  2\|\Psi_{L,M_\infty}\|_{L^{5/4}(Q)} |Q|^{2/5} C_2^3 C_{f,M_\infty} \rho \|v\|_{L^2(Q)}^2  = \bar c_1\rho \|v\|_{L^2(Q)}^2.\]

The estimate for the term $I_3$ follows in a similar way. Indeed, using H\"older's inequality, Assumption (A3), estimates \eqref{E10} and \eqref{E11}, and Lemma \ref{L2.8}, we obtain
\begin{align*}
I_3\leq& C_{f,M_\infty} \|\bar \varphi\|_{L^{5/4}(Q)}\|z_{u,v}-z_v\|_{L^{10}(Q)}\|z_{u,v}+z_v\|_{L^{10}(Q)}\\
\leq & C_{f,M_\infty} M_{5/4}(1 + \|\bar\mu_Q\|_{\mathcal{M}(Q)} + \|\bar\mu_\Omega\|_{\mathcal{M}(\Omega)}) \|z_{u,v}-z_v\|_{L^{10}(Q)} \|z_{u,v}+z_v\|_{L^{10}(Q)}\\
\leq & 2 |Q|^{\frac{2}{5}} C_{f,M_\infty}^2 M_{5/4}(1 + \|\bar\mu_Q\|_{\mathcal{M}(Q)} + \|\bar\mu_\Omega\|_{\mathcal{M}(\Omega)}) C_2^3\rho\|v\|_{L^2(Q)}^2 = \bar c_3\rho \|v\|_{L^2(Q)}^2.
\end{align*}

To estimate $I_5$, we denote $\phi = \varphi_u-\bar\varphi$. This function satisfies the linear equation
\[
\left\{\begin{array}{rl}
-\displaystyle\frac{\partial \phi}{\partial t} + A^* \phi \, + \displaystyle\frac{\partial f}{\partial y} (x,t,\bar y) \phi  = & \displaystyle\frac{\partial L}{\partial y} (x,t,y_u)- \displaystyle\frac{\partial L}{\partial y} (x,t,\bar y)
\\& + \displaystyle\left(\frac{\partial f}{\partial y} (x,t,\bar y)-\frac{\partial f}{\partial y} (x,t,y_u)\right)\varphi_u\mbox{ in }Q,\\
 \phi = & 0 \mbox{ on }\Sigma,\\
  \phi(T) = &0\mbox{ in }\Omega.
  \end{array}
  \right.
\]
Then, noticing that $5/4<(n+2)/n$, we deduce from Lemma \ref{L2.5}, the Mean Value Theorem, assumptions (A2) and (A3) and Lemma \ref{L2.8} that $ \varphi_u-\bar\varphi\in L^{5/4}(Q)$ and
\begin{align*}
\|\varphi_u-&\bar\varphi\|_{L^{5/4}(Q)}\leq K_{5/4}( \|\Psi_{L,M_\infty}\|_{L^1(Q)} + C_{f,M_\infty} \|\varphi_u\|_{L^1(Q)}) \rho\\
& \le K_{5/4}( \|\Psi_{L,M_\infty}\|_{L^1(Q)} + C_{f,M_\infty} M_1 (1 + \|\bar\mu_Q\|_{\mathcal{M}(Q)} + \|\bar\mu_\Omega\|_{\mathcal{M}(\Omega)}) \rho
= \bar C\rho
\end{align*}
 Using H\"older's inequality, Assumption (A3), the above estimate, and {Lemma \ref{L2.3}} we have
\begin{align*}
  I_5\leq  & C_{f,M_\infty} \|\varphi_u-\bar\varphi\|_{L^{5/4}(Q)} \|z_{u,v}\|^2_{L^{10}(Q)}
\leq
 C_{f,M_\infty} \bar C C_2^2\rho\|v\|_{L^2(Q)}^2 = \bar c_5\rho \|v\|_{L^2(Q)}^2.
\end{align*}

To estimate the second term, we use the continuity property of the second derivative of $L$ assumed in (A2). Given
\[\epsilon_2 = \frac{\varepsilon}{5|Q|^{4/5}C_2^2},\]
there exists $\delta_2 >0$ such that, if $\|y_u-\bar y\|_{C(\bar Q)} < \delta_2$, then
\[\|\frac{\partial^2 L}{\partial y^2}(x,t,\bar y)-\frac{\partial^2 L}{\partial y^2}(x,t,y_u)\|_{L^\infty( Q)} < \epsilon_2.
\]
Using this, H\"older's inequality and Lemma \ref{L2.3}, we conclude
\[
I_2\leq
\epsilon_2 \|z_{u,v}\|^2_{L^2(Q)} \leq
\epsilon_2 |Q|^{4/5}\|z_{u,v}\|^2_{L^{10}(Q)} \leq
\epsilon_2 |Q|^{4/5}C_2^2\|v\|^2_{L^{2}(Q)}  =
\frac{\varepsilon}{5}\|v\|^2_{L^{2}(Q)}.
\]

The estimation of $I_4$ is similar. Now we use the continuity property of the second derivative of $f$ with respect to $y$ assumed in (A3). Given
\[\epsilon_4 = \frac{\varepsilon}{5M_{5/4}(1 + \|\bar \mu_Q\|_{\mathcal{M}(Q)} + \|\bar \mu_\Omega\|_{\mathcal{M}(\Omega)}) C_2^2},\]
there exists $\delta_4>0$ such that
\[\|\frac{\partial^2 f}{\partial y^2}(x,t,\bar y)-\frac{\partial^2 f}{\partial y^2}(x,t,y_u)\|_{L^\infty( Q)} < \epsilon_4
\]
if $\|y_u-\bar y\|_{C(\bar Q)} < \delta_4$.
Using that $\bar\varphi\in L^{5/4}(Q)$, lemmas \ref{L2.3} and \ref{L2.8}, and H\"older's inequality, we obtain
\begin{align*}
I_4 = & \int_Q\left|\left(\frac{\partial^2 f}{\partial y^2}(x,t,y_u)-\frac{\partial^2 f}{\partial y^2}(x,t,\bar y)\right)\bar\varphi z_{u,v}^2\right|\dx\dt
\leq  \epsilon_4 \|\bar\varphi\|_{L^{5/4}(\Omega)}\|z_{u,v}\|^2_{L^{10}(\Omega)}\\
\leq &  \epsilon_4 C_2^2 M_{5/4} (1 + \|\bar \mu_Q\|_{\mathcal{M}(Q)} + \|\bar \mu_\Omega\|_{\mathcal{M}(\Omega)}) \|v\|^2_{L^2(\Omega)} = \frac{\varepsilon}{5} \|v\|_{L^2(Q)}^2.
\end{align*}

Putting all five estimates together and taking
\begin{equation}\label{E12}
\rho = \min\left\{\frac{\varepsilon}{5\bar c_1},\delta_2,
\frac{\varepsilon}{5\bar c_3},\delta_4,\frac{\varepsilon}{5\bar c_5},1\right\}
\end{equation}
 we get the desired estimate.
\end{proof}

\begin{remark}\label{ZZR2}
In \cite[Theorem 2.8]{Casas-Troltzsch2014}, the authors prove that $\bar\varphi$, the adjoint state related to a control $\bar u$ that satisfies first order optimality conditions, belongs to $L^\infty(\Omega)$ for all $\nu\geq 0$. Nevertheless, in our case, we have not been able to get such a regularity result. This would allow us to obtain an inequality similar to that of Theorem \ref{T2.9}, but with $\varepsilon \|v\|_{L^2(Q)}^2$ replaced by $\varepsilon \|z_{\bar u,v}\|_{L^2(Q)}^2$, which is the key to write second order sufficient optimality conditions in the case $\nu=0$.
As a consequence, the second order analysis remains an open problem for the case $\nu=0$.
\end{remark}

\section{First order necessary optimality conditions}\label{S3}
Existence of a global solution of \Pb can be proved by standard methods under the assumption of existence of an admissible control $u\in\uad$.

For $s\in [1,+\infty]$, we say that $\bar u\in\uad$ is a local solution of \Pb in the sense of $L^s(Q)$ if there exists $\varepsilon > 0$ such that
\[J(\bar u) \leq J(u)\ \forall u\in\uad\text{ such that }\|u-\bar u\|_{L^s(Q)} \leq  \varepsilon.\]
If $\uad$ is bounded in $L^\infty(Q)$, which is assumed for $n=2$ or $3$, and if $\bar u$ is a local solution in the $L^r(Q)$-sense for some $r\in[1,+\infty)$, then it is also a local solution in the $L^s(Q)$ sense for all $s\in[1,+\infty]$; cf. \cite[Section 5]{Casas-Mateos2017}. However, a local solution in the sense of $L^\infty(Q)$ is not necessarily a local solution in $L^r(Q)$ for $r<+\infty$.
For $n=1$, if $\alpha = -\infty$ or $\beta = +\infty$, if $\bar u$ is a local solution in the $L^{r}(Q)$ sense, then it is a local solution in the $L^s(Q)$ sense for all $r\in[s,+\infty]$.
We simply say that $\bar u\in\uad$ is a local solution of \Pb if it is a local solution in the sense of $L^s(Q)$ for some $s\in[1,+\infty]$.

In the rest of the paper we will often use the notation
\[U_{\alpha,\beta} = \{u\in L^2(Q):\ \alpha\leq u(x,t)\leq\beta\ \text{ for a.e. }(x,t)\in Q\}.\]
Due to our assumptions on $\alpha$ and $\beta$, we have that $U_{\alpha,\beta}\subset L^\infty(Q)$ if $n=2$ or $n=3$, which is crucial to have the differentiability of $\mathcal{L}$ at every point of $\uad$. For $n=1$ the Lagrangian is of class $C^2$ in $L^2(Q)$.
First order conditions can be deduced from \cite[Theorem 5.2]{Casas1993}; see also \cite{Casas-1997}.

\begin{theorem}\label{T3.1}
  Let $\bar u\in\uad$ be a local solution of \Pb.
  Assume that the following linearized Slater condition holds: there exists $u_0\in U_{\alpha,\beta}$ such that
   \begin{equation}\label{E13}
      y_{\bar u}(x,t) + z_{\bar u, u_0-\bar u}(x,t) < \gamma\qquad \forall (x,t)\in \bar Q.
   \end{equation}
  Then, there exist $\bar y \in  L^2(0,T;H^1_0(\Omega))\cap C(\bar Q)$, $\bar\varphi \in L^r(0,T;W^{1,s}(\Omega)) \cap L^q(Q)$ for all $r,s\in [1,2)$ with $2/r+n/s>n+1$ and all $q<(n+2)/n$, and nonnegative measures $\bar\mu_Q\in\mathcal{M}(Q)$ and $\bar \mu_\Omega\in\mathcal{M}(\Omega)$ such that
\begin{equation}\label{E14}
\left\{\begin{array}{rcll}
\displaystyle\frac{\partial \bar y}{\partial t} + A \bar y + f(x,t,\bar y) &=& \bar u&\mbox{ in }Q,\\
 \bar y& =& 0 &\mbox{ on }\Sigma,\\
  \bar y(0)& = &y_0&\mbox{ in }\Omega,
  \end{array}
  \right.
\end{equation}
\begin{equation}\label{E15}
\left\{\begin{array}{rcll}
-\displaystyle\frac{\partial\bar\varphi}{\partial t} + A^* \bar\varphi + \displaystyle\frac{\partial f}{\partial y}(x,t,\bar y)\bar\varphi &=& \displaystyle\frac{\partial L}{\partial y}(x,t,\bar y) +\bar \mu_Q& \mbox{ in }Q,\\
 \bar\varphi& =& 0&\mbox{ on }\Sigma,\\
 \bar\varphi(\cdot,T) & = & \bar\mu_\Omega&\mbox{ in }\Omega,
  \end{array}
  \right.
\end{equation}
\begin{equation}\label{E16}
  \supp \bar\mu_Q\subset\{(x,t)\in Q:\ \bar y(x,t) = \gamma\},\qquad
  \supp \bar\mu_\Omega\subset\{x\in \Omega:\ \bar y(x,T) = \gamma\},
\end{equation}
\begin{equation}\label{E17}
  \int_Q(\bar\varphi+\nu\bar u)(u-\bar u)\dx\dt\geq 0\ \forall u\in U_{\alpha,\beta}.
\end{equation}
\end{theorem}
 Since $\nu>0$, it is well known that \eqref{E17} is equivalent to the pointwise projection formula
\[\bar u(x) = \mathrm{proj}_{[\alpha,\beta]}\left(-\frac{1}{\nu}\bar\varphi(x)\right).\]

\section{Sufficient optimality conditions}\label{S4}
For a given triplet $(\bar u,\bar \mu_Q,\bar \mu_\Omega)\in\uad\times\mathcal{M}_Q\times\mathcal{M}_\Omega$ and any $\tau\geq 0$, we define the cone

\begin{align*}
C^\tau_{\bar u,\bar \mu} = \{ v\in L^p(Q):&
\ \frac{\partial\mathcal{L}}{\partial u}(\bar u,\bar\mu_Q,\bar\mu_\Omega)v\leq \tau \|v\|_{L^p(Q)},\\
& v(x,t)\geq 0\text{ if }\bar u(x,t)=\alpha,\ v(x,t)\leq 0\text{ if }\bar u(x,t)=\beta,\\
& z_v(x,t) \leq\tau\|v\|_{L^p(Q)}\text{ if }\bar y(x,t)=\gamma\text{ and }\\
& \int_Q z_v \mathrm{d}\bar\mu_Q+\int_\Omega z_v(\cdot,T) \mathrm{d}\bar\mu_\Omega\geq -\tau \|v\|_{L^p(Q)}\}.
\end{align*}

This cone is an extension of the following one, introduced in \cite{Casas-delosReyes-Troltz-2008-SIOPT}:

\begin{align*}
C_{\bar u,\bar \mu} = \{ v\in L^p(Q):&
\ \frac{\partial\mathcal{L}}{\partial u}(\bar u,\bar\mu_Q,\bar\mu_\Omega)v=0,\\
& v(x,t)\geq 0\text{ if }\bar u(x,t)=\alpha,\ v(x,t)\leq 0\text{ if }\bar u(x,t)=\beta,\\
& z_v(x,t) \leq 0\text{ if }\bar y(x,t)=\gamma\text{ and }\\
& \int_Q z_v \mathrm{d}\bar\mu_Q+\int_\Omega z_v(\cdot,T) \mathrm{d}\bar\mu_\Omega = 0\}.
\end{align*}
Indeed, if  $(\bar u,\bar \mu_Q,\bar \mu_\Omega)\in\uad\times\mathcal{M}_Q\times\mathcal{M}_\Omega$ is a triplet satisfying first order optimality conditions, we have that
\[C_{\bar u,\bar \mu} = C^0_{\bar u,\bar \mu} \subset C^\tau_{\bar u,\bar \mu}\ \forall \tau >0.\]
In  \cite{Dunn98} it was proved that the second order condition based on the critical cone of the type $C_{\bar u,\bar\mu}$ is not enough for a sufficient optimality conditions in some cases. Hence an extension of the cone was suggested.

Finally, we state and prove the main result of the paper.

\begin{theorem}\label{T4.1}
  Let $(\bar u,\bar \mu_Q,\bar \mu_\Omega)\in\uad\times\mathcal{M}_Q\times\mathcal{M}_\Omega$ be a triplet satisfying the first order optimality conditions \eqref{E14}--\eqref{E17} and suppose that there exist $\tau> 0$ and $\delta >0$ such that

  \begin{equation}\label{E18}
  \begin{array}{ll}
 \displaystyle\frac{\partial^2\mathcal{L}}{\partial u^2}(\bar u,\bar \mu_Q,\bar \mu_\Omega)v^2 > 0\ \forall v\in C_{\bar u,\bar \mu} \setminus\{0\}&\text{ if } n =1,\\ \\
  \displaystyle\frac{\partial^2\mathcal{L}}{\partial u^2}(\bar u,\bar \mu_Q,\bar \mu_\Omega)v^2\geq \delta \|v\|^2_{L^2(Q)}\ \forall v\in C^\tau_{\bar u,\bar \mu}&\text{ if } n = 2\text{ or }n=3.
  \end{array}
  \end{equation}
  Then, there exist $\varepsilon >0$ and $\kappa >0$ such that
  \[J(\bar u) +\frac{\kappa}{2}\|u-\bar u\|^2_{L^2(Q)}\leq J(u)\ \forall u\in\uad\text{ such that }\|u-\bar u\|_{L^2(Q)} \leq \varepsilon.\]
\end{theorem}
\begin{proof}
  We will proceed by contradiction. Suppose that the statement is false. Then, there exists a sequence $\{u_k\}_{k\geq 1} \subset \uad$ such that
\begin{equation}\label{E19}
  \|u_k-\bar u\|_{L^2(Q)}<\frac{1}{k}\quad \text{and} \quad J(u_k) < J(\bar u) + \frac{1}{2k}\|u_k-\bar u\|_{L^2(Q)}^2\quad \forall k \ge 1.
  \end{equation}
Taking into account Assumption (A5) and noting that $\uad\subset U_{\alpha,\beta} \subset L^\infty(\Omega)$ if $n=2$ or $n=3$, we can define for every $k$
  \[\rho_k = \|u_k-\bar u\|_{L^p(Q)}\text{ and }v_k = \frac{1}{\rho_k}(u_k-\bar u).\]
  It is clear that $\|v_k\|_{L^p(Q)} = 1$ for all $k$ and hence we can extract a subsequence, denoted in the same way, such that $v_k\rightharpoonup v$ in $L^p(Q)$. The proof is split in five steps.

  \underline{\em Step 1:} $v\in C_{\bar u,\bar \mu}$. Let us check that $v$ satisfies the four conditions that characterize the functions in the cone.

  {\em Condition 1.}  We have that $\bar y(x,t)=\gamma$ in $\supp\bar\mu_Q \cup\supp\bar\mu_\Omega$ and, since $u_k$ is admissible, $y_{u_k}(x,t)\leq \gamma$  $\forall (x,t)$ in $\bar Q$, so
  \begin{align*}
  \int_Q y_{u_k}\mathrm{d}\bar\mu_Q +\int_\Omega y_{u_k}\mathrm{d}\bar\mu_\Omega \leq
  \int_Q \gamma\mathrm{d}\bar\mu_Q +\int_\Omega \gamma\mathrm{d}\bar\mu_\Omega
  =  \int_Q \bar y\mathrm{d}\bar\mu_Q +\int_\Omega \bar y\mathrm{d}\bar\mu_\Omega.
  \end{align*}
Using \eqref{E19}, we deduce
  \[
  \mathcal{L}(u_k,\bar\mu_Q,\bar\mu_\Omega) < \mathcal{L}(\bar u,\bar\mu_Q,\bar\mu_\Omega) + \frac{1}{2k}\|u_k-\bar u\|_{L^2(Q)}^2.
  \]
  By the Mean Value Theorem, for every $k$ we can find a measurable function $\theta_k$ with $0\leq\theta_k(x,t)\leq1$ such that, denoting $u_{\theta_k} = u_k+\theta_k(u_k-\bar u)$, we get
  \begin{equation}
  \label{E20}\partial_u \mathcal{L}(u_{\theta_k},\bar\mu_Q,\bar\mu_\Omega)(u_k-\bar u)<\frac{1}{2k}\|u_k-\bar u\|_{L^2(Q)}^2.
  \end{equation}
  With H\"older's inequality we infer
  \[\|u_k-\bar u\|_{L^2(Q)}^2\leq \rho_k^2|Q|^{\frac{p-2}{p}}.\]
  Thus, dividing  inequality \eqref{E20} by $\rho_k$ we obtain
  \[
  \partial_u \mathcal{L}(u_{\theta_k},\bar\mu_Q,\bar\mu_\Omega)v_k<\frac{1}{2k}\frac{\|u_k-\bar u\|_{L^2(Q)}^2}{\rho_k} \leq
  \frac{1}{2k}|Q|^{\frac{p-2}{p}}\rho_k.
  \]
  On the right hand side, we have that $\rho_k\leq(\beta-\alpha)^{\frac{p-2}{p}}\|u_k-\bar u\|_{L^2(Q)}^{2/p}\to 0$. To pass to the limit in the left hand side, we first notice that $u_{\theta_k}\to \bar u$  and $v_k\rightharpoonup v$ in $L^p(Q) \subset L^2(Q)$. Therefore
  \[\int_Q\nu u_{\theta_k}v_k\dx\dt\to\int_Q\nu\bar u v\dx\dt.\]
  Next, using Theorem \ref{T2.1} and Assumption (A3), we deduce from the convergence $u_{\theta_k}\to \bar u$ in $L^p(Q)$ that $y_{u_{\theta_k}}\to \bar y$ in $C(\bar Q)$. From this convergence, Theorem \ref{T2.1}, Assumption (A3), and Lemma \ref{L2.5} we infer that $\varphi_{u_{\theta_k},\bar\mu}\to \bar \varphi$ in $L^q(Q)$ for all $q<(n+2)/n$. From the conditions imposed on the choice of the exponent $p$ in Assumption (A5), we have in particular that $\varphi_{u_{\theta_k}}\to \bar \varphi$ in $L^{p'}(Q)$, which implies
  \[\int_Q \varphi_{u_{\theta_k},\bar\mu} v_k\dx\dt \to \int_Q \bar \varphi v\dx\dt.\]
  Therefore, we deduce that
   \[\partial_u \mathcal{L}(\bar u,\bar\mu_Q,\bar\mu_\Omega)v\leq 0.\]
   Since $u_k\in U_{\alpha,\beta}$, from the first order optimality condition \eqref{E17} and the expression for the derivative of $\mathcal{L}$ given in Corollary \ref{C2.7}, we have that $\displaystyle\frac{\partial \mathcal{L}}{\partial u}(\bar u,\bar\mu_Q,\bar\mu_\Omega)v_k\geq 0$ for all $k$, and hence
  \begin{equation}\label{E21}\displaystyle\frac{\partial \mathcal{L}}{\partial u}(\bar u,\bar\mu_Q,\bar\mu_\Omega)v\geq 0,\end{equation}
so we have that the first condition is satisfied.

  {\em Condition 2.}  Since the $v_k$ satisfy the sign conditions, taking weak limits and noting that the set of $L^p(\Omega)$-functions satisfying the sign conditions is a closed and convex set, we have that so does $v$ satisfy the sign conditions.

  {\em Condition 3.} If $\bar y(x,t) = \gamma$, due to the admissibility of $u_k$, we have that
  \[\zeta_k(x,t) = \frac{y_{u_k}(x,t)-\bar y(x,t)}{\rho_k}\leq 0.\]
  The function $\zeta_k$ satisfies the equation
\[\left\{\begin{array}{rcll}
\displaystyle\frac{\partial\zeta_k}{\partial t} + A \zeta_k + \displaystyle\frac{\partial f}{\partial y}(x,t, y_{\theta'_k})\zeta_k &=& v_k& \mbox{ in }Q,\\
 \zeta_k& =& 0&\mbox{ on }\Sigma,\\
 \zeta_k & = & 0&\mbox{ in }\Omega,
  \end{array}
  \right.
\]
where $y_{\theta'_k} =  y_{u_k}+\theta'_k\cdot(\bar y-y_{u_k})$ and $\theta'_k$ is a measurable function such that $0\leq\theta'_k\leq 1$ a.e. in $Q$.
The weak convergence $v_k\rightharpoonup v$ in $L^p(Q)$ and the strong convergence $y_{\theta'_k} \to \bar y$ in $C(\bar Q)$ imply the strong convergence $\zeta_k\to z_v$ in $C(\bar Q)$, so we have that $z_v(x,t)\leq 0$ for a.a. $(x,t)\in Q$ such that $\bar y(x,t) = \gamma$.

{\em Condition 4.} From \eqref{E21}
we deduce
    \[\displaystyle\frac{\partial \mathcal{L}}{\partial u}(\bar u,\bar\mu_Q,\bar\mu_\Omega)v = J'(\bar u)v+\int_Q z_v\mathrm{d}\bar\mu_Q +\int_\Omega z_v\mathrm{d}\bar\mu_\Omega  \ge  0\]
Using \eqref{E19} and proceeding as for Condition 1, we have that $J'(\bar u)v\leq 0$. Since $z_v\leq 0$ in $\supp\bar\mu_Q\cup\supp\bar\mu_\Omega$,  the above inequality is possible only if $J'(\bar u)v = 0$ and
\[\int_Q z_v\mathrm{d}\bar\mu_Q +\int_\Omega z_v\mathrm{d}\bar\mu_\Omega  =  0.\]

\underline{\em Step 2:} Let us see that, indeed, $v=0$.

Using \eqref{E19} and performing a second order Taylor expansion, we deduce that for all $k$, there exists a measurable function $\theta_k''$ with $0\leq\theta_k''\leq 1$  such that
\[\rho_k\frac{\partial\mathcal{L}}{\partial u}(\bar u, \bar\mu_Q,\bar\mu_\Omega)v_k +
\frac{\rho_k^2}{2}\frac{\partial^2\mathcal{L}}{\partial u^2}(u_{\theta_k''}, \bar\mu_Q,\bar\mu_\Omega)v_k^2
\leq \frac{1}{2k}\|u_k-\bar u\|_{L^2(Q)}^2,\]
where $u_{\theta_k''} = u_k+\theta_k''(\bar u-u_k)$. Using that ${\partial_u \mathcal L}(\bar u, \bar\mu_Q,\bar\mu_\Omega)v_k\geq 0$ and dividing by $\rho_k^2/2$, we obtain
\begin{equation}\label{E22}
\frac{\partial^2\mathcal{L}}{\partial u^2}(u_{\theta_k''}, \bar\mu_Q,\bar\mu_\Omega)v_k^2 \leq \frac{1}{k}\|v_k\|_{L^2(Q)}^2.
\end{equation}
Using this estimate we are going to prove that $\frac{\partial^2\mathcal{L}}{\partial u^2}(\bar u, \bar\mu_Q,\bar\mu_\Omega)v^2 \leq 0$. This inequality, the fact that $v\in C_{\bar u, \bar\mu}\subset C^\tau_{\bar u, \bar\mu}$, and condition \eqref{E18} imply that $v=0$. To get the desired inequality we proceed as follows
\begin{align}
  \frac{\partial^2\mathcal{L}}{\partial u^2}(\bar u, \bar\mu_Q,\bar\mu_\Omega)v^2  \leq &
  \liminf_{k\to\infty} \frac{\partial^2\mathcal{L}}{\partial u^2}(\bar u, \bar\mu_Q,\bar\mu_\Omega)v_k^2  \nonumber \\
  \leq &\limsup_{k\to\infty} \frac{\partial^2\mathcal{L}}{\partial u^2}(u_{\theta_k''}, \bar\mu_Q,\bar\mu_\Omega)v_k^2 \label{MyE23}\\
  & +
  \lim_{k\to \infty}
\left(  \frac{\partial^2\mathcal{L}}{\partial u^2}(\bar u, \bar\mu_Q,\bar\mu_\Omega)v_k^2-
  \frac{\partial^2\mathcal{L}}{\partial u^2}(u_{\theta_k''}, \bar\mu_Q,\bar\mu_\Omega)v_k^2\right)  \nonumber
\end{align}
Since $\|v_k\|_{L^2(Q)}^2 \le |Q|^{\frac{p-2}{p}}\|v_k\|^2_{L^p(Q)} = |Q|^{\frac{p-2}{p}}$, we conclude from
\eqref{E22} that
\begin{equation}
\label{MyE24}
\limsup_{k\to\infty} \frac{\partial^2\mathcal{L}}{\partial u^2}(u_{\theta_k''}, \bar\mu_Q,\bar\mu_\Omega)v_k^2 \leq
\lim_{k\to\infty}\frac{1}{k}|Q|^{\frac{p-2}{p}} = 0.
\end{equation}
The strong convergence  $u_{\theta''_k}\to \bar u$ in $L^p(Q)$ and the last statement in Theorem \ref{T2.1} imply the convergence of $y_{u_{\theta''_k}}$ to $\bar y$ in $C(\bar Q)$. Using this fact, we can deduce from Theorem \ref{T2.9}, that for every $\varepsilon >0$ there exists $k_\varepsilon>0$ such that
\begin{equation}\label{MyE25}
\left|\frac{\partial^2\mathcal{L}}{\partial u^2}(\bar u, \bar\mu_Q,\bar\mu_\Omega)v_k^2-\frac{\partial^2\mathcal{L}}{\partial u^2}( u_{\theta''_k}, \bar\mu_Q,\bar\mu_\Omega)v_k^2\right|\leq \varepsilon \|v_k\|_{L^2(Q)}^2\quad \forall k\geq k_\varepsilon.
\end{equation}
Using again that $\|v_k\|_{L^2(Q)}^2\leq |Q|^{\frac{p-2}{p}}$ the above inequality implies that the last limit in \eqref{MyE23} is $0$, and gathering \eqref{E22}, \eqref{MyE23} and \eqref{MyE24} we conclude that $\frac{\partial^2\mathcal{L}}{\partial u^2}(\bar u, \bar\mu_Q,\bar\mu_\Omega)v^2
\leq 0$.

\underline{\em Step 3:} Achieving a contradiction in the case $n =1$.

From Theorem \ref{T2.1} and Assumption (A3), we notice that the weak convergence  $v_k\rightharpoonup 0$ in $L^p(Q)$ implies that $z_{v_k}\to 0$ in $C(\bar Q)$. Therefore, from the expression of the second derivative of the Lagrangian,
\begin{align*}
\frac{\partial^2\mathcal{L}}{\partial u^2}&(u_{\theta_k''}, \bar\mu_Q,\bar\mu_\Omega)v_k^2\\& =
\int_Q\left(\frac{\partial^2 L}{\partial y^2}(x,t,y_{u_{\theta_k''}})-\frac{\partial^2 f}{\partial y^2}(x,t,y_{u_{\theta_k''}})\varphi_{u_{\theta_k''}}\right)z_{v_k}^2\dx\dt+ \nu \|v_k\|_{L^2(Q)}^2,
\end{align*}
we can deduce that $\|v_k\|_{L^2(Q)}\to 0$. In the case $n=1$, we have that $p=2$; see Assumption (A5). So this is in contradiction with the fact that $\|v_k\|_{L^p(Q)}=1$, and the proof is complete in this case.

\underline{\em Step 4:} Let us check that there exists $k_0>0$ such that $v_k\in C^\tau_{\bar u, \bar\mu}$ for all $k\geq k_0$.

Since $\displaystyle\frac{\partial \mathcal{L}}{\partial u}(\bar u, \bar\mu_Q,\bar\mu_\Omega)v_k\to 0$ and $\|v_k\|_{L^p(Q)} = 1$, it is clear that for $k$ big enough,
\[\displaystyle\frac{\partial \mathcal{L}}{\partial u}(\bar u, \bar\mu_Q,\bar\mu_\Omega)v_k\leq \tau =  \tau\|v_k\|_{L^p(Q)}.\]
Moreover, for every $k$ the function $v_k$ satisfies trivially the sign condition imposed in the definition of $C^\tau_{\bar u,\bar\mu}$ due to $u_k\in U_{\alpha,\beta}$.
Finally, since $z_{v_k}\to 0$ in $C(\bar Q)$, we have that
\[\int_Q z_{v_k}\mathrm{d}\bar\mu_Q + \int_\Omega z_{v_k}\mathrm{d}\bar\mu_\Omega\to 0.\]
Hence, for $k$ big enough, it is clear that both $z_{v_k}(x,t) \leq \tau = \tau\|v_k\|_{L^p(Q)}$ if $\bar y(x,t)=\gamma$ and
\[
\int_Q z_{v_k}\mathrm{d}\bar\mu_Q + \int_\Omega z_{v_k}\mathrm{d}\bar\mu_\Omega\geq -\tau = -\tau\|v_k\|_{L^p(Q)}.
\]
The above arguments prove the existence of $k_0$ such that $v_k\in C^\tau_{\bar u, \bar\mu}$ for all $k\geq k_0$.

\underline{\em Step 5:} Achieving a contradiction  if $n=2$ or $n=3$.

Since $v_k\in C^\tau_{\bar u, \bar\mu}$ for $k\geq k_0$,  the second order optimality condition \eqref{E18} implies
\[
\frac{\partial^2\mathcal{L}}{\partial u^2}(\bar u, \bar\mu_Q,\bar\mu_\Omega)v_k^2\geq \delta\|v_k\|_{L^2(Q)}^2\quad \forall k \ge k_0.
\]
Choosing $\varepsilon = \delta/2$ in \eqref{MyE25}, we infer from the previous inequality that
\[\frac{\partial^2\mathcal{L}}{\partial u^2}(u_{\theta''_k}, \bar\mu_Q,\bar\mu_\Omega)v_k^2\geq \frac{\delta}{2}\|v_k\|_{L^2(Q)}^2\quad \forall k\geq \max\{k_0,k_{\delta/2}\}.\]
Finally, combining \eqref{E22} and the above inequality we deduce
\[
\frac{1}{k}\geq\frac{\delta}{2}\quad \forall k>\max\{k_0,k_{\delta/2}\},
\]
which is clearly false.
\end{proof}

\begin{remark}\label{R3}
Notice that through all the paper we have not used explicitly the assumption $\nu >0$ for the case $n > 1$. All the results in the paper would also be true if $\nu=0$ and $n > 1$, but in this case Theorem 4.1 can be vacuous. Indeed, it is shown at the end of Section 2 in \cite{Casas2012} that if $\nu=0$, then condition \eqref{E18} does not hold except maybe in exceptional cases.
\end{remark}

\begin{remark}\label{ZZR4}
In the case $n=1$, \eqref{E18}  is the same condition used in \cite[Theorem 7.5]{Casas-delosReyes-Troltz-2008-SIOPT} to obtain local optimality in the  $L^\infty(Q)$ sense.
We can also prove that \eqref{E18} is equivalent to the existence of $\tau>0$ and $\delta>0$ such that
\begin{equation}\label{E23} \displaystyle\frac{\partial^2\mathcal{L}}{\partial u^2}(\bar u,\bar \mu_Q,\bar \mu_\Omega)v^2\geq \delta \|v\|^2_{L^2(Q)}\ \forall v\in C^\tau_{\bar u,\bar \mu}.\end{equation}
It is obvious that \eqref{E23} implies \eqref{E18}. To see the other implication, suppose that \eqref{E23} is false. Then, for $k=1,2,\ldots$ there exists $v_k\in C^{1/k}_{\bar u,\bar \mu}$ with $\|v_k\|_{L^2(Q)}=1$ such that
\begin{equation}\label{E24}
  \displaystyle\frac{\partial^2\mathcal{L}}{\partial u^2}(\bar u,\bar \mu_Q,\bar \mu_\Omega)v_k^2\leq \frac{1}{k}.
\end{equation}
Since $\{v_k\}$ is bounded in $L^2(Q)$, there exists $v\in L^2(Q)$ such that $v_k\rightharpoonup v$ weakly in $L^2(Q)$. Using this weak convergence and the strong convergence of $z_{v_k}$ to $z_v$ in $C(\bar Q)$, it is immediate that $v\in C_{\bar u,\bar\mu}$. On the other hand, taking the lower limit in \eqref{E24} we obtain
\[\displaystyle\frac{\partial^2\mathcal{L}}{\partial u^2}(\bar u,\bar \mu_Q,\bar \mu_\Omega)v^2\leq 0,\]
so condition \eqref{E18} implies that $v=0$ and, consequently, $z_{v_k}\to 0$ in $C(\bar Q)$. Finally, using that $\|v_k\|_{L^2(Q)}=1$, we obtain
\begin{align*}
0\geq & \displaystyle\frac{\partial^2\mathcal{L}}{\partial u^2}(\bar u,\bar \mu_Q,\bar \mu_\Omega)v^2 = \lim_{k\to\infty}
\frac{\partial^2\mathcal{L}}{\partial u^2}(\bar u, \bar\mu_Q,\bar\mu_\Omega)v_k^2\\
 = & \lim_{k\to\infty}\Big\{
\int_Q\left(\frac{\partial^2 L}{\partial y^2}(x,t,\bar y)-\frac{\partial^2 f}{\partial y^2}(x,t,\bar y)\bar\varphi\right)z_{v_k}^2\dx\dt+ \nu \|v_k\|_{L^2(Q)}^2\Big\}= \nu
\end{align*}
which is a contradiction with the assumption $\nu >0$.
\end{remark}

\section{Bilateral constraints}\label{S5}
All the results of the paper apply with the usual changes if the set $\uad$ is replaced by
\[
\uad =\{u \in L^\infty(Q): \alpha \le u(x,t) \le \beta\ \text{ and }\gamma_{\min} \leq y_u(x,t)\leq\gamma_{\max} \text{ for a.a. } (x,t) \in Q\},
\]
where $\gamma_{\min} < 0 < \gamma_{\max}$. First order conditions read as follows.
\begin{theorem}\label{T5.1}
  Let $\bar u\in\uad$ be a local solution of \Pb.
  Assume that the following linearized Slater condition holds: there exists $u_0\in U_{\alpha,\beta}$ such that
   \[
      \gamma_{\min} < y_{\bar u}(x,t) + z_{\bar u, u_0-\bar u}(x,t) < \gamma_{\max}\qquad \forall (x,t)\in \bar Q.
   \]
  Then, there exist $\bar y \in  L^2(0,T;H^1_0(\Omega))\cap C(\bar Q)$, $\bar\varphi \in L^r(0,T;W^{1,s}(\Omega)) \cap L^q(Q)$ for all $r,s\in [1,2)$ with $2/r+n/s>n+1$ and all $q<(n+2)/n$, and measures $\bar\mu_Q\in\mathcal{M}(Q)$ and $\bar \mu_\Omega\in\mathcal{M}(\Omega)$ such that equations \eqref{E14} and \eqref{E15} are satisfied,
\begin{align*}
  \supp \bar\mu^+_Q\subset\{(x,t)\in Q:\ \bar y(x,t) = \gamma_{\max}\},\ &
    \supp \bar\mu^-_Q\subset\{(x,t)\in Q:\ \bar y(x,t) = \gamma_{\min}\},
  \\
  \supp \bar\mu^+_\Omega\subset\{x\in \Omega:\ \bar y(x,T) = \gamma_{\max}\},\ &
    \supp \bar\mu^-_\Omega\subset\{x\in \Omega:\ \bar y(x,T) = \gamma_{\min}\},
\end{align*}
where
$\bar \mu_Q = \bar \mu_Q^+-\bar\mu_Q^-$ and $\bar \mu_\Omega = \bar \mu_\Omega^+-\bar\mu_\Omega^-$ are the Jordan decompositions of $\bar\mu_Q$ and $\bar\mu_\Omega$, and
\[
  \int_Q(\bar\varphi+\nu\bar u)(u-\bar u)\dx\dt\geq 0\ \forall u\in U_{\alpha,\beta}.
\]
\end{theorem}
To formulate second order sufficient optimality conditions, the appropriate cone in this case is the following one:
\begin{align*}
C^\tau_{\bar u,\bar \mu} = \{ v\in L^p(Q):&
\ \frac{\partial\mathcal{L}}{\partial u}(\bar u,\bar\mu_Q,\bar\mu_\Omega)v\leq \tau \|v\|_{L^p(Q)},\\
& v(x,t)\geq 0\text{ if }\bar u(x,t)=\alpha,\ v(x,t)\leq 0\text{ if }\bar u(x,t)=\beta,\\
& z_v(x,t) \leq + \tau\|v\|_{L^p(Q)}\text{ if }\bar y(x,t)=\gamma_{\max},\\
& z_v(x,t) \geq - \tau\|v\|_{L^p(Q)}\text{ if }\bar y(x,t)=\gamma_{\min}\text{ and }\\
& \int_Q |z_v|\, \mathrm{d}|\bar\mu_Q|+\int_\Omega |z_v(\cdot,T)|\, \mathrm{d}|\bar\mu_\Omega|\leq \tau \|v\|_{L^p(Q)}\},
\end{align*}
where $|\bar \mu_Q|$ and $|\bar \mu_\Omega|$ denote the total variation of $\bar \mu_Q$ and $\bar \mu_\Omega$ respectively.

\section*{Funding}

The first and second authors were partially supported by MCIN/ AEI/10.13039/501100011033 under research project PID2020-114837GB-I00.


\end{document}